\documentclass[11pt,a4paper]{article}
\usepackage{latexsym}
\usepackage{amsthm,amssymb,amstext}
\usepackage{amsmath}

\usepackage{graphicx}
\usepackage{psfrag,color}
\usepackage{subfig}

\newcommand{\R}{\mathbb{R}}

\newcommand{\set}[2]{\left\{#1\,\left|\;#2\right.\right\}}

\newcommand{\V}[2]{{\left(\begin{array}{c}#1\\#2\end{array}\right)}}

\def\I{{I}}

\def\h{h}
 \newcommand{\normL}[1]{\left\|#1\right\|_{\I}}
 \newcommand{\normRD}[1]{\left|#1\right|_{\I}}
 \newcommand{\norm}[1]{\left\|#1\right\|}

 \newcommand\restr[2]{{
  \left.\kern-\nulldelimiterspace 
  #1 
  \vphantom{\big|} 
  \right|_{#2} 
  }}

\newtheorem{Theorem}{Theorem}[section]

\newtheorem{Lemma}[Theorem]{Lemma}

\newtheorem{Corollary}[Theorem]{Corollary}

\theoremstyle{definition}

\newtheorem{Assumption}[Theorem]{Assumption}

\numberwithin{equation}{section}

\def\mathref#1{\ifmmode\mathrm{(\ref{#1})}\else(\ref{#1})\fi}

\def\rhx2{\sqrt{1+ r_{h,x}^2}}

\def\Uad{{U_\textup{ad}}}
\def\uopt{{\bar u}}
\def\yopt{{\bar y}}
\def\popt{{\bar p}}
\def\uoptk{{\bar u_k}}
\def\yoptk{{\bar y_k}}

\def\V{{H^1_0(\Omega)}}
\def\Vd{{H^{-1}(\Omega)}}
\def\H{{L^2(\Omega)}}
\def\L2{{L^2(I,\H)}}

\title{Crank-Nicolson time stepping and variational discretization of control-constrained parabolic optimal control problems}
\author{
Nikolaus von Daniels\footnote{Schwerpunkt Optimierung und Approximation,
Universit\"at Hamburg, Bundesstra{\ss}e 55, 20146 Hamburg, Germany.},
Michael Hinze\footnotemark[1], and 
Morten Vierling\footnotemark[1]
}

\date{6-MAR-2015}

\begin{document}  \nocite{*}

\maketitle

\begin{center}
 {\bf \LARGE  }
 \end{center}

 {\small {\bf Abstract:}
We consider a control constrained parabolic optimal control problem and use variational discretization for its time semi-discretization. The state equation is treated with a Petrov-Galerkin scheme using a piecewise constant Ansatz for the state and piecewise linear, continuous test functions. This results in variants of the Crank-Nicolson scheme for the state and the adjoint state. Exploiting a superconvergence result we prove second order convergence in time of the error in the controls. Moreover, the piecewise linear and continuous parabolic projection of the discrete state on the dual time grid provides a second order convergent approximation of the optimal state without further numerical effort. Numerical experiments confirm our analytical findings.}\\[2mm]

{\small {\bf AMS Subjects Classification:} 49J20, 35K20, 49M05, 49M25, 49M29, 65M12, 65M60.} \\[2mm]

\pagenumbering{arabic}
\section{Introduction}\label{S:Intro}
The purpose of this paper is to prove optimal a priori error bounds for the variational time semi-discretization of a generic parabolic optimal control problem, where the state in time is approximated with a Petrov Galerkin scheme. The key idea consists in choosing piecewise linear, continuous test functions and a discontinuous, piecewise constant Ansatz for the approximation of the state equation. With this Petrov Galerkin Ansatz variational discretization of the optimal control problem delivers a cG(1) time approximation of the optimal time semi-discrete adjoint state. The resulting time integration schemes for the state and the adjoint state are variants of the Crank-Nicolson scheme. Combining this setting with the supercloseness result of Corollary \ref{C:Supercloseness} for interval means we are able to prove second order in time convergence of the time discrete optimal control in Theorem \ref{T:main}. Moreover, the piecewise linear and continuous parabolic projection of the discrete state based on the the values of the discrete state on the dual time grid provides a second order convergent approximation of the optimal state without further numerical effort, see Theorem \ref{T:2ndorderstate}.

Our work is motivated by the work \cite{MeidnerVexler2011} of Meidner and Vexler, whoms technical results we use whenever possible.
Under mild assumptions on the active set they show the same convergence order in time for the post-processed piecewise linear, continuous parabolic projection of the piecewise constant in time optimal control.
This control is obtained by a Petrov Galerkin scheme with variational discretization of the parabolic optimal control problem. In comparison to their work, we switch Ansatz and test space in our numerical schemes.

Our work is novel in several aspects:
\begin{itemize}
\item In \eqref{E:WFD} to the best of the authors knowledge a new, fully variational time-discretization scheme for the parabolic state equation is presented. It results in a Crank-Nicolson scheme with initial damping step for the nodal values of the state.
\item In Theorem \ref{T:main} we provide an optimal error estimate for the time-discrete control, namely
\[
 \norm{\uopt-\uoptk}_{L^2(0,T,\mathbb R^D)} \le Ck^2,
\]
with $\uopt$ the optimal control, which is the solution of the optimal control problem \eqref{OCP}, and $\uoptk$ denoting the optimal control obtained from the related discretized problem \eqref{OCPDiscr} (see below). Here, $k$ denotes the grid size of the time grid. This result could be compared to \cite[Theorem 6.2]{MeidnerVexler2011}. There, under mild assumptions on the structure of the active set w.r.t. $\uopt$, a similar bound is obtained for the post-processed parabolic projection of a piecewise constant optimal control. Our approach avoids such an assumption in the numerical analysis, and presents an error estimate for the variational-discrete optimal control.
\item In Theorem \ref{T:2ndorderstate} we prove 
\[
\norm{\yopt-\pi_{P_k^*}\yoptk}_{L^2(0,T,L^2(\Omega))} \le Ck^2,
\]
where $\pi_{P_k^*}\yoptk$ denotes the piecewise linear and continuous parabolic projection of the discrete state based on the the values of the discrete state on the dual time grid defined in Section \ref{S:Tdiscr}. Since these values are already known, the projection is for free.
\item Our approach demonstrates that variational discretization of \cite{Hinze2005} through the choice of Ansatz and test space in a Petrov Galerkin approximation of parabolic optimal control problems offers the possibility to specify the discrete structure of variational optimal controls. Of course this feature applies also to other classes of PDE constrained optimal control problems.
\end{itemize}

In our note we only consider semi-discretization in time for two reasons. On the one hand, error estimation for standard spatial finite element approximations of the time-semidiscrete optimal control problem \eqref{OCPDiscr} is along the lines of \cite[Section 6.2]{MeidnerVexler2011}. On the other hand, we are interested in the approximation of optimal controls, which in a realistic time-dependent scenario only depend on time, see the possible definitions of the control operator $B$ below.\\

With $I:=(0,T) \subset \mathbb{R}$, $T<\infty$, and a fixed function $y_d\in\L2$, we consider the linear-quadratic optimal control problem
\begin{equation}\label{OCP}\tag{$\mathbb P$}
\begin{aligned}
&\min_{ y\in Y,u\in \Uad} J(y,u)=\frac{1}{2}\|y-y_d\|^2_{\L2}+\frac{\alpha}{2}\|u\|^2_U,\\
&\text{s.t. } y=S(Bu,y_0).
\end{aligned}
\end{equation}
The state space $Y$ is given by
\[ 
    Y:= W(I):=\{v\in L^2(I,\V), \partial_t v\in L^2(I,\Vd)\}
           \hookrightarrow C([0,T],L^2(\Omega)),
\]
and the operator $S: L^2(I,\Vd)\times L^2(\Omega) \rightarrow W(I)$, $(f,\kappa)\mapsto y:= S(f,\kappa)$, denotes the weak solution operator associated with the parabolic problem
\begin{equation}\label{E:State}
\begin{aligned} \partial_t y -\Delta y &= f &&\text{in }I\times\Omega\,,\\
y&=0&&\text{in } I\times\partial\Omega\,,\\
y(0)&=\kappa &&\text{in } \Omega,
\end{aligned}
\end{equation}
i.e. for $(f,\kappa) \in L^2(I,\Vd)\times L^2(\Omega)$ the function $y\in W(I)$ satisfies
$y(0)=\kappa$ and
\begin{multline}\label{E:WF}
\int\limits_0^T \langle \partial_t y(t),v(t)\rangle_{\Vd\V} + a(y(t),v(t))\, dt
= \int\limits_0^T \langle f(t),v(t)\rangle_{\Vd\V}\, dt\\
    \quad\forall\, v\in L^2(I,\V).
\end{multline}
Here $\Omega\subset \R^n$, $n=2,3$, is a convex polygonal domain with boundary $\partial \Omega$, and for $y,v\in\V$ we define
\[
a(y,v):= \int\limits_\Omega \nabla y(x) \nabla v(x)\ dx.
\]
In what follows several choices of control spaces are feasible. In particular we may choose as  control space $U=L^2(I,\R^D)$, $D\in\mathbb{N}$, and as the admissible set
\[ \Uad=\set{u\in U}{a_i\le u_i(t)\le b_i\text{ a.e. in $I$, }i=1,\dots, D},\] where $a_i$, $b_i\in\R$, $a_i < b_i$ $(i=1,\dots, D)$. In this case the control operator is given by
\begin{equation}\label{E:B}
B: U\rightarrow L^2(I,\Vd)\,,\quad u\mapsto
 \left( t\mapsto \sum_{i=1}^Du_i(t)g_i \right),
\end{equation}
where $g_i\in\Vd$ are given functionals, whose regularity is specified in Assumption \ref{A:Regularity} below. Clearly, $B$ is linear and bounded. A further possible choice for the control space is $U=\L2$ with 
\[ 
   \Uad=\set{u\in U}{a\le u(t,x)\le b\text{ a.e. in }I\times\Omega}, 
\]
where $a<b$ denote real constants. In this case the control operator $B$ is the injection from $U$ into $L^2(I,H^{-1}(\Omega))$. In both cases the admissible set $\Uad$ is closed and convex. We build our exposition upon the practical more relevant first choice of time-dependent amplitudes as controls.

It is well known that the operator $S$ is well defined, i.e. for every $(f,\kappa) \in L^2(I,\Vd)\times L^2(\Omega)$ a unique state $y \in W(I)$ satisfying \eqref{E:WF} exists. Furthermore, it fulfills
\[
\|y\|_{W(I)} \le C \left\{\|f\|_{L^2(I,\Vd)}+\|\kappa\|_{L^2(\Omega)}\right\}.
\]
Now let $y \in Y$ denote the unique solution of \eqref{E:WF}, and let $v\in W(I)$. Then it follows from integration by parts for functions in $W(I)$, that with the bilinear form $A:W(I)\times W(I)\rightarrow \mathbb{R}$ defined by
\begin{equation}\label{bilinA}
A(y,v):= \int\limits_0^T -\langle \partial_t v(t),y(t)\rangle_{\Vd\V} 
  + a(y(t),v(t))\, dt + (y(T),v(T))_{\H}\,, 
\end{equation}
the state $y$ also satisfies
\begin{equation}\label{E:WFM}
A(y,v) = \int\limits_0^T \langle f(t),v(t)\rangle_{\Vd\V}\, dt + (\kappa,v(0))_{\H}
    \quad\forall\ v\in W(I).
\end{equation}
Furthermore, $y$ is the only function in $Y$ which satisfies \eqref{E:WFM}. In the next section we use the bilinear form $A$ to define our numerical approximation scheme for the state equation.


With $\mathcal O(k^2)$ error-bounds for the control in mind we follow \cite{MeidnerVexler2011} and make the following assumptions on the data.
\begin{Assumption}\label{A:Regularity}
Let $y_d\in H^1(I,L^2(\Omega))$, and $y_d(T)\in \V$. Let further $g_i\in\V$, $i=1,\dots, D$, and finally $y_0\in\V$ with $\Delta y_0\in\V$.
\end{Assumption}

A lot of literature is available on optimal control problems with parabolic state equations. We refer to \cite{HPUU09} for a comprehensive discussion, and also to \cite{ApelFlaig2012,MeidnerVexler08a,MeidnerVexler08b,MeidnerVexler2011,SpringerVexler2013} for the most recent developments related to optimal control with Galerkin methods in time.

The paper is organized as follows. In section 2 we briefly summarize the solution theory of the optimal control problem. In section 3 we analyse the regularity of the state and the adjoint state, which plays an important role in the time discretization. In section 4 the time discretization of state and adjoint state is discussed in detail. In section 5 we introduce variational discretization of the optimal control problem $(P)$ and prove second order convergence of the variational discrete controls in time. In section 6 we present numerical results which confirm our analytical findings.

\section{The continuous problem $(\mathbb{P})$}
It is well known that problem \eqref{OCP} admits a unique solution $(\yopt,\uopt)\in Y\times U$, where $\yopt = S(B\uopt,y_0)$. Moreover, using the orthogonal projection $P_\Uad: L^2(I,\R^D) \rightarrow \Uad$, the optimal control is characterized by the first-order necessary and sufficient condition
 \begin{equation}\label{FONC}
\uopt = P_\Uad\left(-\frac{1}{\alpha}B'\bar p\right),
 \end{equation}
where $(\bar p,\bar q)\in L^2(I,\V)\times \H$ (here we use reflexivity of the involved spaces) denotes the adjoint variable which is the unique solution to
\begin{multline}\label{E:WAdj}
\int\limits_0^T \langle \partial_t \tilde y(t),\bar p(t)\rangle_{\Vd\V} + a(\tilde y(t),\bar p(t))\, dt + (\tilde y(0),\bar q)_{\H}\\
 = \int\limits_0^T\int\limits_\Omega (\yopt(t,x)-y_d(t,x)) \tilde y(t,x)\, dxdt \quad \forall\ \tilde y \in W(I).
\end{multline}
Here, $B': L^2(I,\V) \rightarrow L^2(I,\R^D)$ denotes the adjoint operator of $B$, which is characterized by
 \begin{equation}
 \label{E:Badj}
 B'q(t)=\left(\langle g_1,q(t)\rangle_{\Vd\V}\ ,\ \dots\ ,\ \langle g_D,q(t)\rangle_{\Vd\V}\right)^T.
 \end{equation}
Furthermore we note that for $v\in L^2(I,\R^D)$ there holds
\[
P_\Uad(v)(t) = \left(P_{[a_i,b_i]}(v_i(t))\right)_{i=1}^D,
\]
where for $a,b,z\in\mathbb{R}$ with $a\le b$ we set $P_{[a,b]}(z):=\max\{a,\min\{z,b\}\}$.

Since $\yopt-y_d \in L^2(I,L^2(\Omega))$ in \eqref{E:WAdj}, we have $\bar p\in W(I)$, so that by integration by parts for functions in $W(I)$ we conclude from \eqref{E:WAdj} (compare \eqref{E:WFM})
\begin{multline}\label{E:AA}
\int\limits_0^T -\langle \partial_t \bar p(t),\tilde y(t)\rangle_{\Vd\V} + a(\tilde y(t),\bar p(t))\, dt + \\
(\tilde y(0),\bar q)_{\H}+ (\tilde y(T),\bar p(T))_{\H} - (\tilde y(0),\bar p(0))_{\H} \\
 =\int\limits_0^T\int\limits_\Omega (\yopt(t,x)-y_d(t,x)) \tilde y(t,x)\, dxdt\quad\forall\ \tilde y\in W(I),
\end{multline}
so that the function $\bar p$ can be identified with the unique weak solution to the adjoint equation
\begin{equation}\label{E:Adj}
 \begin{aligned} -\partial_t \bar p -\Delta \bar p &= h &&\text{in }I\times\Omega\,,\\
\bar p&=0&&\text{on } I\times\partial\Omega\,,\\
\bar p(T)&=0 &&\text{on } \Omega, \end{aligned}
\end{equation}
with $h:=\bar y-y_d$. Moreover, $\bar q=\bar p(0).$


\section{Regularity results}
In this section we summarize some existence and regularity results concerning equation \eqref{E:State} and \eqref{E:Adj}, which can also be found in e.g. \cite{MeidnerVexler2011}. We abbreviate
\[
\normL{ \cdot}:=    \|\cdot\|_{L^2(I,L^2(\Omega))}\,,\quad\normRD{\cdot}:=\|\cdot\|_{L^2(I,\R^D)}.
\]
For the unique weak solutions $y$ to \eqref{E:State} and $p$ to \eqref{E:Adj} we have from \cite[Theorems 7.1.5 and 5.9.4]{Evans1998} the regularity results.
\begin{Lemma}\label{L:LowReg}
For $f,h\in L^2(\I,\H)$ and $\kappa\in \V$ the solutions $y$ of \eqref{E:State} and $p$ of \eqref{E:Adj} satisfy
\[
y,p\in L^2(\I,H^2(\Omega)\cap\V)\cap H^1(\I,\H)\hookrightarrow C([0,T],\V).
\]
Furthermore, with some constant $C>0$ there holds
\[
\normL{y}+\normL{ \partial_t y }+\normL{\Delta y}+\max_{t\in\bar\I}\|y(t)\|_{H^1(\Omega)}\le C \left\{\normL{f}+\|\kappa\|_{H^1(\Omega)}\right\},
\]
and
\[
\normL{ \partial_t p }+\normL{\Delta p}+\max_{t\in\bar\I}\|p(t)\|_{H^1(\Omega)}\le C \normL{h}.
\]
\end{Lemma}

However, in order to achieve $\mathcal O (k^2)$-convergence we need more regularity, i.e., at least second weak time derivatives. From \cite[Proposition 2.1]{MeidnerVexler2011} we have

\begin{Lemma}\label{L:HighReg}
Let $f,h\in H^1(\I,\H)$, $f(0),\h(T)\in \V$, and $\kappa\in\V$ with $\Delta \kappa\in\V$. Then the solutions y of \eqref{E:State} and p of \eqref{E:Adj} satisfy
\[
y,p\in H^1(\I,H^2(\Omega)\cap\V)\cap H^2(\I,\H).
\]
With some constant $C>0$ we have the a priori estimates
\[
\normL{ \partial^2_t y }+\normL{\partial_t\Delta y}\le C\left\{ \| f\|_{H^1(\I,L^2(\Omega))}+\| f(0)\|_{H^1(\Omega)}+\|\kappa \|_{H^1(\Omega)}+\|\Delta\kappa \|_{H^1(\Omega)}\right\},
\]
and
\[
\normL{ \partial^2_t p }+\normL{\partial_t\Delta p}\le C\left\{ \|\h\|_{H^1(\I,L^2(\Omega))}+\|\h(T)\|_{H^1(\Omega)}\right\}.
\]
\end{Lemma}

From Lemma \ref{L:LowReg} we conclude that the optimal state $\yopt$ lives in
$H^1(\I,\H)$, and $\yopt(T)\in\V$.
Thus, by Lemma \ref{L:HighReg} and Assumption \ref{A:Regularity} the optimal adjoint state $\popt$ is an element of $H^2(\I,\H)$. It then follows from \eqref{E:Badj} that $B'\popt \in H^2(\I,\R^D)$. Furthermore, for $v\in W^{1,r}(\I,\R^D)$, $1\le r\le\infty$, one has
\[
\|\partial_t P_{\Uad}(v)\|_{L^r(\I,\R^D)}\le   \|\partial_t v\|_{L^r(\I,\R^D)},
\]
so that \eqref{FONC} and $B'\popt \in H^2(\I,\R^D)$ imply $\uopt\in W^{1,\infty}(\I,\R^D)$.

Hence, using our Assumption \ref{A:Regularity}, Lemma \ref{L:HighReg} is applicable to the solution of the state equation and one obtains the following result, see e.g. \cite[Proposition 2.3]{MeidnerVexler2011}.

\begin{Lemma}
Let Assumption \ref{A:Regularity} hold. For the unique solution $(\yopt,\uopt)$ of \eqref{OCP} and the corresponding adjoint state $\popt$ there holds
\[
\yopt,\popt\in H^1(\I,H^2(\Omega)\cap\V)\cap H^2(\I,\H)\,,\quad\text{and}\quad \uopt\in W^{1,\infty}(\I,\R^D)\,.
\]
\end{Lemma}

\section{Time discretization}\label{S:Tdiscr}
\def\j{{m}}\def\K{{M}}
Let $[0,T)=\bigcup_{\j=1}^\K I_m$, where the intervals $I_\j=[t_{\j-1},t_\j)$ are defined through the partition $0=t_0 <  t_1 < \dots < t_\K=T$.
Furthermore, let $t_m^*=\frac{t_{\j-1}+t_\j}{2}$ for $m=1,\dots,\K$ denote the interval midpoints. 
By $0=:t_0^* < t_1^* < \dots < t_\K^* < t_{\K+1}^*:=T$ we get the so-called \emph{dual partition} of $[0,T)$, namely $[0,T)=\bigcup_{\j=1}^{\K+1} I_m^*$, with $I_\j^*=[t_{\j-1}^*,t_\j^*)$.
The grid width of the first (primal) partition is defined by the mesh-parameters $k_m=t_m-t_{m-1}$ and
\[ k=\max_{1\le\j\le\K} k_\j. \]

On these partitions we define the Ansatz and test spaces of our Petrov Galerkin scheme for the numerical approximation of the optimal control problem \eqref{OCP} w.r.t. time. We set
\[
P_k:=\set{v\in C([0,T],\V)}{\restr{v}{I_\j}\in \mathcal P_1(I_\j,\V)}\hookrightarrow W(I),
\]
\[
P_k^*:=\set{v\in C([0,T],\V)}{\restr{v}{I_\j^*}\in \mathcal P_1(I_\j^*,\V)}\hookrightarrow W(I)
\]
and
\[
Y_k:=\set{v:[0,T]\rightarrow \V}{\restr{v}{I_\j}\in \mathcal P_0(I_\j,\V)}\,.
\]
Here, $\mathcal P_i(J,\V)$, $J\subset \bar I$, $i\in\{0,1\}$, denotes the set of polynomial functions in time of degree at most $i$ on the interval $J$ with values in $\V$.
We note that functions in $P_k\cup Y_k$ can be uniquely determined by $\K+1$ elements from $\V$.
Furthermore each function in $Y_k$ is also an element of $L^2(I,\V)$.\\
In what follows we frequently use the interpolation operators
\begin{enumerate}
\item $\mathcal P_{Y_k}:L^2(I,\V)\rightarrow Y_k$
\[ \restr{\mathcal  P_{Y_k} v}{I_\j}:=\frac{1}{k_\j}\int_{t_{\j-1}}^{t_\j} v dt \text{ for } m=1,\dots,\K, \text{ and } \mathcal P_{Y_k}v(T):=0\]
\item $\Pi_{Y_k} : C([0,T],\V)\rightarrow Y_k$
\[ \restr{\Pi_{Y_k} v }{I_\j}:= v\left(t_m^*\right) \text{ for } m=1,\dots,\K, \text{ and } \Pi_{Y_k}v(T) := v(T).\]
\item $\pi_{P_k^*} : C([0,T],\V)\cup Y_k\rightarrow P_k^*$
\[
\begin{aligned}
\restr{\pi_{P_k^*} v }{I_1^*\cup I_2^*} &:= v(t_1^*) + \frac{t-t_1^*}{t_2^*-t_1^*} (v(t_2^*)-v(t_1^*)),\\
\restr{\pi_{P_k^*} v }{I_\j^*} &:= v\left(t_{\j-1}^*\right) + \frac{t-t_{\j-1}^*}{t_\j^*-t_{\j-1}^*} (v(t_\j^*)-v(t_{\j-1}^*)), \text{ for } m=3,\dots,\K-1,\\
\restr{\pi_{P_k^*} v }{I_\K^*\cup I_{\K+1}^*} &:= v(t_{\K-1}^*) + \frac{t-t_{\K-1}^*}{t_\K^*-t_{\K-1}^*} (v(t_\K^*)-v(t_{\K-1}^*)).
\end{aligned}
\]

\end{enumerate}

To apply variational discretization to \eqref{OCP} we next introduce the
Petrov-Galerkin scheme for the approximation of the states. For this purpose we extend the bilinear form $A$ of \eqref{bilinA} from $W(I)$ to $W(I)\cup Y_k$, i.e. we consider $A$ as a mapping $A:W(I)\cup Y_k \times W(I)\rightarrow \mathbb{R}$. Then, according to \eqref{E:WFM} we for $(f,\kappa) \in L^2(I,\Vd)\times L^2(\Omega)$ consider the time-semidiscrete problem: Find $y_k\in Y_k$, such that
\begin{equation}\label{E:WFD}
A(y_k,v_k)= \int\limits_0^T \langle f(t),v_k(t)\rangle_{\Vd\V}\, dt +(\kappa,v_k(0))_{\H}
     \quad\forall\ v_k\in P_k.
\end{equation}
Then $y_k\in Y_k$ is uniquely determined. This follows from the fact that with
\[
y_k = \alpha_{M+1}\chi_{\{T\}} + \sum\limits_{i=1}^{M} \alpha_i \chi_{I_i}, \quad \alpha_i \in \V \text{ for } i=1,\dots,M+1,
\]
the coefficients $\alpha_i$ for $i=2,\dots, M$ are determined by a Crank-Nicolson
 scheme with a (Rannacher) smoothing step \cite{RR84} for $\alpha_1$, and
 $\alpha_{M+1}$ is uniquely determined by $\alpha_M$.

Note that in all of the following results $C$ denotes a generic, strict positive
 real constant that does not depend on quantities which appear to the right of it.

The following stability result will be useful in the later analysis.

\begin{Lemma}\label{L:yDiscrStab}
  Let $y_k\in Y_k$ solve \eqref{E:WFD} for $f\in \L2$ and $\kappa\in\H$ given.
 Then there exists a constant $C>0$ independent of the time mesh size $k$ such that
  \[
  \normL{y_k} \le C \left( \normL{f} + \norm{\kappa}_{\H} \right).
  \]
\end{Lemma}
\begin{proof}
We test in \eqref{E:WFD} with the $v_k\in P_k$ that is uniquely determined by
$v_k(T):=0$ and $\restr{\partial_t v_k}{I_m}:=-\restr{y_k}{I_m}$, $m=1,\dots,M$.
We get using integration by parts in $W(I)$
\begin{equation*}
\begin{aligned}
  A(y_k,v_k)
  &= \normL{y_k}^2 - \int\limits_0^T 
       (\partial_t\nabla v_k,\nabla v_k)_{\H}\, dt\\
  &= \normL{y_k}^2 + \frac 12 \left( \norm{\nabla v_k(0)}_{\H}^2
              -\norm{\nabla v_k(T)}_{\H}^2  \right)\\
  &= \normL{y_k}^2 +\frac 12 \norm{\nabla v_k(0)}_{\H}^2
     \overset{\eqref{E:WFD}}{=}
     \int\limits_0^T (f,v_k)_{\H}\, dt + (\kappa,v_k(0))_{\H}\\
  &\le \frac 12\left( \normL{y_k}^2+\norm{\nabla v_k(0)}_{\H}^2\right) 
      + C\left(\normL{f}^2 + \norm{\kappa}_{\H}^2 \right),
\end{aligned}
\end{equation*}
where we use the Cauchy-Schwarz inequality and Poincar\'{e}'s
inequality. Rearranging terms and taking the square root yields the claim.
\end{proof}

For Petrov Galerkin approximations $y_k \in Y_k$ of states $y\in W(I)$ we can only expect $\mathcal O(k)$ convergence, since $y_k$ is piecewise constant in time, compare \cite[Lemma 5.2]{MeidnerVexler2011}. In order to obtain $\mathcal O(k^2)$ control approximations in our convergence analysis for problem \eqref{OCP} we rely on the following super-convergence results for the projections $\Pi_{Y_k}$ and $\mathcal P_{Y_k}$, see e.g. \cite[Lemma 5.3]{MeidnerVexler2011}.
\begin{Lemma}\label{L:Supercloseness} Let $(f,\kappa)$ satisfy the regularity requirements of Lemma \ref{L:HighReg}, and let $y,y_k$ solve \eqref{E:WF} and \eqref{E:WFD} with data $(f,\kappa)$, thus $y\in H^1\left(\I,H^2(\Omega)\cap\V\right)\cap H^2\left(\I,\H\right)$. Then there holds
\[
\normL{y_k-\Pi_{Y_k}y}\le C k^2\left(\normL{\partial_t^2y}+\normL{\partial_t\Delta y}\right).
\]
\end{Lemma}
Note that the proof of {\cite[Lemma 5.3]{MeidnerVexler2011}} is applicable in our situation since the initial value $\kappa$ is the same for both, the continuous problem \eqref{E:WF} and \eqref{E:WFD}.

\begin{Corollary}\label{C:Supercloseness}
Let the assumptions of Lemma \ref{L:Supercloseness} hold. Then there holds
\[
\normL{y_k-\mathcal P_{Y_k}y}\le C k^2\left(\normL{\partial_t^2y}+\normL{\partial_t\Delta y}\right)\,.
\]
\end{Corollary}
  \begin{proof}
 With the result of Lemma \ref{L:Supercloseness} at hand it suffices to show that
 \begin{equation}\label{E:ProjEst}
 \normL{\Pi_{Y_k}y-\mathcal P_{Y_k}y}\le  k^2\normL{\partial_t^2y}
 \end{equation}
holds. We prove this estimate for smooth functions $w\in C^2(\I,\H)\cap H^2(\I,\H)$. The result then follows by a density argument.

Suppose $w\in C^2(\I,\H)\cap H^2(\I,\H)$. We use the Taylor expansion of $w$ at $t_m^*$ and obtain
\begin{multline}
\left\|\int_{t_{\j-1}}^{t_\j}w(t)-w(t^*_\j)dt\right\|_\H^2
=\left\|\int_{t_{\j-1}}^{t_\j}(t-t^*_\j)\partial_t w(t^*_\j) + \int_{t^*_\j}^t (t-s)\partial_t^2 w(s)ds dt\right\|_\H^2\\
\le k_\j \int_{t_{\j-1}}^{t_\j}\left\| \int_{t^*_\j}^t (t-s)\partial_t^2 w(s)ds \right\|_\H^2 dt
\le k_\j^4\int_{t_{\j-1}}^{t_\j} \int_{t^*_\j}^t \left\|\partial_t^2 w(s) \right\|_\H^2 ds dt\\
\le k^5_\j\int_{t_{\j-1}}^{t_\j}  \left\|\partial_t^2 w(s) \right\|_\H^2 ds,
\end{multline}

where we have used the Cauchy-Schwarz inequality twice. This proves
\[
\left(\sum_{\j=1}^\K k_\j\left\|\frac{1}{k_\j}\int_{t_{\j-1}}^{t_\j}w(t)-w(t^*_m) dt\right\|_\H^2\right)^{\frac{1}{2}}\le  k^2\normL{\partial_t^2w},
\]
which is \eqref{E:ProjEst}.
\end{proof}

For the next Lemma, see \cite[Lemma 5.6]{MeidnerVexler2011}, we need the following condition on the time grid:
\begin{Assumption}\label{A:grid}
  There exist constants $0 < c_1 \le c_2 < \infty$ independent of $k$ such that
  \[ c_1 \le \frac{k_m}{k_{m+1}} \le c_2 \]
  holds for all $m=1,2,\dots,\K-1$.
\end{Assumption}

\begin{Lemma}\label{L:yprojprop}
  Let the Assumption \ref{A:grid} be fulfilled. The interpolation operator $\pi_{P_k^*}$ has the following properties, where $C>0$ in both cases denotes a constant independent of $k$.
  \begin{enumerate}
    \item $\normL{w-\pi_{P_k^*}w} \le C k^2 \normL{\partial_t^2 w} 
            \quad\forall\ w\in H^2(I,\H)$,
    \item $\normL{\pi_{P_k^*}w_k} \le C \normL{w_k} 
            \quad\forall\ w_k\in Y_k$.
  \end{enumerate}
\end{Lemma}

Since the state is discretized by piecewise constant functions, we can only expect first order convergence in time for its discretization error. The following Lemma shows that a projected version of the discretized state converges second order in time to the continuous state. The benefit of this result will be discussed in the numerics section.

\begin{Lemma}\label{L:convyproj}
Let $y$ and $y_k$ be given as in Lemma \ref{L:Supercloseness}. Then there holds
\[ \normL{\pi_{P_k^*}y_k - y} \le C k^2 \left( \normL{\partial_t^2 y} + \normL{\partial_t \Delta y}\right). \]
\end{Lemma}
\begin{proof}
Making use of the splitting
\[
\normL{\pi_{P_k^*}y_k - y}  = \normL{\pi_{P_k^*}\left( y_k - \Pi_{Y_k} y\right)}  + \normL{\pi_{P_k^*}\Pi_{Y_k} y - y}
= \normL{\pi_{P_k^*}\left( y_k - \Pi_{Y_k} y\right)}  + \normL{\pi_{P_k^*} y - y},
\]
the claim is an immediate consequence of Lemma \ref{L:yprojprop} and Lemma \ref{L:Supercloseness}.
\end{proof}

In the numerical treatment of problem \eqref{OCP} we also need error estimates for discrete adjoint functions $p_k\in P_k \hookrightarrow W(I)$. For $h \in L^2(I,\Vd)$ we consider the problem: Find $p_k\in P_k$ such that
\begin{equation}\label{E:AdjDiscr}
A(\tilde y,p_k)
=\int\limits_0^T \langle h(t),\tilde y(t)\rangle_{\Vd\V}\, dt
\quad\forall\ \tilde y \in Y_k.
\end{equation}
This problem admits a unique solution $p_k \in P_k$. 
This follows from the fact that if we write
\[
p_k(t) = \sum\limits_{i=0}^M \beta_i b_i(t)
\]
with coefficients $\beta_i \in \V$ and $b_i \in C([0,T])$, 
$b_i(t_j)=\delta_{ij}$, for $i,j=0,\dots,M$, the coefficients $\beta_i$ are
determined by a backward in time Crank-Nicolson scheme, starting with 
$\beta_M\equiv 0$. Similar to \cite[Lemma 4.7]{MeidnerVexler2011} we have
 the following stability result, which we need to prove Theorem 
\ref{T:2ndorderstate}.
\begin{Lemma}\label{L:AdjDiscrStab}
  Let $p_k\in P_k$ solve \eqref{E:AdjDiscr} with $\h\in \L2$.
 Then there exists a constant $C>0$ independent of $k$ such that
  \[
      \norm{p_k}_{H^1(I,\H)}+\norm{p_k(0)}_{H^1(\Omega)} \le C\normL{\h}.
  \]
\end{Lemma}
\begin{proof}
We define $\tilde y\in Y_k$ by     
$\restr{\tilde y}{I_m}:=-\restr{\partial_t p_k}{I_m}$, $m=1,\dots,M$, 
with $\tilde y(T)\in\V$ arbitrary. Testing with $\tilde y$ in  
\eqref{E:AdjDiscr} we obtain
\[
A(-\partial_t p_k,p_k) = \frac 12\norm{\nabla p_k(0)}^2_{\H}
   + \normL{\partial_t p_k}^2
     \overset{\eqref{E:AdjDiscr}}{=} 
    \int\limits_0^T (h,-\partial_t p_k)_{\H}\, dt\,,
\]
where we have used $p_k(T)=0$. Arguing as in Lemma \ref{L:yDiscrStab}
delivers the desired result.
\end{proof}

Moreover, from \cite[Lemma 6.3]{MeidnerVexler2011} we have the following convergence results for discrete adjoint approximations.
\begin{Lemma}\label{L:AdjConvergence}
Let $p,p_k$ solve \eqref{E:Adj} and \eqref{E:AdjDiscr}, respectively, where $h\in \L2$. Then there holds
\[
\normL{p_k-p}\le C k^2\left(\normL{\partial_t^2 p} + \normL{\partial_t\Delta p} \right).
\]
\end{Lemma}
One essential ingredient of our convergence analysis is given by the following result.
\begin{Lemma}\label{L:ImprRate}
Let $y$ and $y_k$ as in Lemma \ref{L:Supercloseness}, and let $p_k(h)\in P_k$ 
denote the solution to \eqref{E:AdjDiscr} with right hand side $h$. Then 
there holds
\[ 
            \normL{p_k(y_k-y)}\le 
          C k^2 \left(\normL{\partial_t^2y}+\normL{\partial_t\Delta y}\right).
\]
\end{Lemma}
\begin{proof}
The function $p_k(y_k-y)$ solves \eqref{E:AdjDiscr} with $p_k(T)=0$ and $h=y_k-y$. Since the test functions $\tilde y$ are elements of $Y_k$ we by Galerkin orthogonality obtain the same solution $p_k$ with right hand side $h=y_k-\mathcal P_{Y_k}y$, i.e. $p_k(y_k-y)=p_k(y_k-\mathcal P_{Y_k}y)$. Hence by Lemma \ref{L:AdjDiscrStab} and Corollary \ref{C:Supercloseness} we obtain
  \[
  \normL{p_k(y_k-y)}=\normL{p_k(y_k-\mathcal P_{Y_k}y)}\le C\normL{y_k-\mathcal P_{Y_k}y}\le C k^2 \left(\normL{\partial_t^2y}+\normL{\partial_t\Delta y}\right),
\]
which is the claim.
\end{proof}

  \section{Variational discretization of the optimal control problem \eqref{OCP}}

To approximate the optimal control problem \eqref{OCP} we apply variational discretization of \cite{Hinze2005} w.r.t. time, where the Petrov Galerkin state discretization introduced in the previous section is applied, i.e. we consider the optimal control problem
  \begin{equation}\label{OCPDiscr}\tag{$\mathbb P_k$}
\begin{aligned}
&\min_{y_k\in Y_k,u\in \Uad} J(y_k,u)=\frac{1}{2}\|y_k-y_d\|^2_\L2+\frac{\alpha}{2}\|u\|^2_U,\\
&\text{s.t. } y_k=S_k(Bu,y_0),
\end{aligned}
\end{equation}
where $S_k$ is the solution operator associated to \eqref{E:WFD}. This problem admits a unique solution $(\bar y_k,\bar u_k) \in Y_k\times \Uad$, where $\bar y_k=S_k(B\bar u_k,y_0)$. The first order necessary and sufficient optimality condition for problem \eqref{OCPDiscr} reads
 \begin{equation}\label{FONCDiscr}
\bar u_k= P_\Uad\left(-\frac{1}{\alpha}B'\bar p_k\right),
 \end{equation}
where $\bar p_k \in P_k$ denotes the discrete adjoint variable, which is the unique solution to
\eqref{E:AdjDiscr} with $h:=\yoptk-y_d$.
Equation \eqref{FONCDiscr} is amenable to numerical treatment although the controls are not discretized explicitly, see \cite{Hinze2005}. It is possible to implement a globalized semismooth Newton strategy in order to solve \eqref{FONCDiscr} numerically, see \cite{HinzeVierling2012}.

First let us establish an error estimate that resembles the standard estimate for variationally discretized problems. To begin with we for $v\in U$ set $y(v):=S(Bv,y_0)$ and denote with $y_k(v)$ the solution to \eqref{E:WFD} with $f:=Bv$. Furthermore, we for $h \in L^2(I,\Vd)$ denote with $p_k(h)$ the solution to \eqref{E:AdjDiscr}.
\begin{Lemma}\label{L:ConvEst}
  Let $\uopt$ and $\uoptk$ solve \eqref{OCP} and \eqref{OCPDiscr}, respectively. Then there holds
  \[
\alpha  \normRD{\uoptk-\uopt}^2\le\left( B'\Big( p_k(\yopt-y_d)-\popt + p_k(y_k(\uopt))-p_k(\yopt) \Big),\uopt-\uoptk\right)_{L^2(\I,\R^D)}.
  \]
  \end{Lemma}
  \begin{proof}
We note that \eqref{FONC} and \eqref{FONCDiscr} can be equivalently expressed as
\begin{align}
      (\alpha \uopt + B'\popt,\uopt-u)_{L^2(\I,\R^D)}
           \le 0 \quad\forall\ u\in \Uad\,,\tag{\ref{FONC}'}\label{E:FONC2}\\
      (\alpha \uoptk + B' \bar p_k,\uoptk-u)_{L^2(\I,\R^D)}
           \le 0 \quad\forall\ u\in \Uad\,.\tag{\ref{FONCDiscr}'}\label{FONCDiscr2}
\end{align}
  Now inserting $\uoptk$ into \eqref{E:FONC2} and $\uopt$ into \eqref{FONCDiscr2} and adding the resulting inequalities yields
\[
  \left(\alpha (\uoptk-\uopt) + B'\Big(\bar p_k-\popt\Big),
         \uoptk-\uopt\right)_{L^2(\I,\R^D)}\le 0.
\]
After some simple manipulations we obtain
\[
\begin{aligned}
\alpha  \normRD{\uoptk-\uopt}^2
  \le&\left( B'\Big(p_k(\yopt-y_d)-\popt+ p_k(y_k(\uopt))-p_k(\yopt)\Big),
             \uopt-\uoptk\right)_{L^2(\I,\R^D)}\\
     &\underbrace{+\left( B'\Big(\bar p_k- p_k(y_k(\uopt)-y_d)\Big),
             \uopt-\uoptk\right)_{L^2(\I,\R^D)}}
          _{-\normL{\bar y_k-y_k(\uopt)}^2} \\ 
  \le&\left( B'\Big(p_k(\yopt-y_d)-\popt+ p_k(y_k(\uopt))-p_k(\yopt)\Big),
             \uopt-\uoptk\right)_{L^2(\I,\R^D)},
\end{aligned}
\]
which is the desired  estimate. \end{proof}
We are now in the position to formulate our main result.
\begin{Theorem}\label{T:main}
Let $\uopt$ and $\uoptk$ denote the solutions to \eqref{OCP} and \eqref{OCPDiscr},
 respectively. Then 
\begin{multline}\label{E:controlestimate}
   \alpha  \normRD{\uoptk-\uopt}\le 
      Ck^2\Big( \norm{\uopt}_{H^1(\I,\R^D)}+\norm{\uopt(0)}_{\mathbb R^D} \\
   +\norm{y_d}_{H^1(\I,\H)}+\norm{y_d(T)}_{H^1(\Omega)}+\norm{y_0}_{H^1(\Omega)}
         +\norm{\Delta y_0}_{H^1(\Omega)}  \Big)
\end{multline}
is satisfied.
\end{Theorem}
\begin{proof}
Making use of the continuity of $B$ and $B'$, compare \eqref{E:B} and
 \eqref{E:Badj}, we directly infer from Lemma \ref{L:ConvEst}
\[
\begin{aligned}
\alpha  \normRD{\uoptk-\uopt}&\le C\left(\normL{p_k(\yopt-y_d)-\popt}+\normL{p_k(y_k(\uopt))-p_k(\yopt)}\right)\\
&\le Ck^2\left( \normL{\partial_t^2\popt}+\normL{\partial_t\Delta \popt}+\normL{\partial_t^2\yopt}+\normL{\partial_t\Delta \yopt} \right).
\end{aligned}
\]
The last estimate follows from the Lemmata \ref{L:AdjConvergence} and \ref{L:ImprRate}. The claim is now a direct consequence of the Lemmata \ref{L:LowReg} and \ref{L:HighReg}.
\end{proof}
Finally we prove second order convergence for $\pi_{P_k^*}\bar y_k$, where we note that this function is obtained for free from $\bar y_k$, since $\bar y_k$ only has to be evaluated on the dual time grid.
\begin{Theorem}\label{T:2ndorderstate} 
Let $\uopt$ and $\uoptk$ denote the solutions to \eqref{OCP} and \eqref{OCPDiscr},
 respectively.  Then there holds
\begin{multline}\label{E:state convergence}
\normL{\bar y-\pi_{P_k^*}\bar y_k} \le
   Ck^2\Big( \normRD{a} + \norm{\uopt}_{H^1(\I,\R^D)}
           +\norm{\uopt(0)}_{\mathbb R^D} \\
   +\norm{y_d}_{H^1(\I,\H)} + \norm{y_d(T)}_{H^1(\Omega)} + 
      \norm{y_0}_{H^1(\Omega)} + \norm{\Delta y_0}_{H^1(\Omega)} \Big).
\end{multline}
\end{Theorem}
\begin{proof}
We have
\[
    \normL{\bar y-\pi_{P_k^*}\bar y_k} \le 
       \normL{\bar y - y(\bar u_k)} + \normL{y(\bar u_k)-\pi_{P_k^*}\bar y_k}.
\]
Lipschitz continuity combined with \eqref{E:controlestimate} yields
\begin{multline*}
\normL{\bar y - y(\bar u_k)} \le 
      Ck^2\Big( \norm{\uopt}_{H^1(\I,\R^D)}+\norm{\uopt(0)}_{\mathbb R^D} \\
   +\norm{y_d}_{H^1(\I,\H)}+\norm{y_d(T)}_{H^1(\Omega)}+\norm{y_0}_{H^1(\Omega)}
         +\norm{\Delta y_0}_{H^1(\Omega)}  \Big).
\end{multline*}
For the second addend we have from Lemma \ref{L:convyproj} combined with
 Lemma \ref{L:HighReg}
\begin{multline*}
\normL{y(\bar u_k)-\pi_{P_k^*}\bar y_k} 
   \le C k^2 \left( \normL{\partial_t^2 y(\bar u_k)} + 
      \normL{\partial_t \Delta y(\bar u_k)}\right) \\ 
   \le C k^2 \left\{ \norm{\bar u_k}_{H^1(\I,\mathbb{R}^D))}+
        \norm{\bar u_k(0)}_{\mathbb R^D}+\norm{y_0}_{H^1(\Omega)}+
       \norm{\Delta y_0}_{H^1(\Omega)}\right\}.
\end{multline*}
Using \eqref{E:Badj} we with the help of \eqref{FONCDiscr}, Lipschitz continuity of the orthogonal projection,
 Lemma \ref{L:AdjDiscrStab}, Lemma \ref{L:yDiscrStab}, 
and \eqref{E:controlestimate} estimate
\begin{equation*}
\begin{aligned}
\norm{\bar u_k}_{H^1(\I,\mathbb{R}^D))}+\norm{\bar u_k(0)}_{\mathbb{R}^D}
  &\le C \left\{\normRD{a} + \norm{\bar p_k}_{H^1(\I,L^2(\Omega))}+
         \norm{\bar p_k(0)}_{H^1(\Omega)}\right\} \\
  &\le C \left\{ \normRD{a} + \normL{\bar y_k-y_d}\right\}\\
  &\le C \left\{ \normRD{a} + 
               \normRD{\uoptk} + \norm{y_0}_{\H} + \normL{y_d}\right\}\\
  &\le C \left\{ \normRD{a} + 
       \norm{\uopt}_{H^1(\I,\R^D)}+\norm{\uopt(0)}_{\mathbb R^D}\right.  \\
  &\left.\quad\quad\quad\quad +\norm{y_d}_{H^1(\I,\H)}+\norm{y_d(T)}_{H^1(\Omega)} \right\}.
\end{aligned}
\end{equation*}
Collecting all estimates gives the desired result.
\end{proof}

\section{Numerical examples}
We now construct numerical examples that validate our main result, i.e. Theorem \ref{T:main}.\\
In both examples we make use of the fact that instead of the linear control operator $B$, given by \eqref{E:B}, we can also use an \emph{affine linear} control operator
\begin{equation}\label{E:Btilde}
\tilde B: U\rightarrow L^2(I,\Vd)\,,\quad u\mapsto g_0 + Bu.
\end{equation}
If we assume that $g_0$ is an element of $H^1(I,L^2(\Omega))$ with initial value $g_0(0)\in\V$, all the preceding theory remains valid.

\subsection{First example}
The first example is taken from \cite{MeidnerVexler2011}. We recall it for convenience in our notation.

Given a space-time domain $\Omega\times I = (0,1)^2 \times (0,0.1)$, i.e. $D=1$, we consider first the control operator $\tilde B$, which is fully characterized by means of the two functions
\[
      g_1(x_1,x_2) := \sin(\pi x_1)\sin(\pi x_2)\,,
\]
\[
      g_0(t,x_1,x_2) := -\pi^4 w_a(t,x_1,x_2) -BP_\Uad\left(
      -\frac{1}{4\alpha}(\exp(a\pi^2 t)-\exp(a\pi^2 T))\right)\,,
\]
where
\[
      w_a(t,x_1,x_2) := \exp(a\pi^2 t)\,\sin(\pi x_1)\sin(\pi x_2)\,,\quad a\in\R\,,
\]
denote eigenfunctions of $\pm \partial_t - \Delta$. As a consequence we have
\[
      (B'z)(t) = \int_\Omega z(t,x_1,x_2)\cdot g_1(x_1,x_2)\, dx_1 dx_2\,,
\]
compare \eqref{E:Badj}. Note that we consider the adjoint of $B$, not of $\tilde B$.
Furthermore we take
\[
      y_d(t,x_1,x_2) := \frac{a^2-5}{2+a}\pi^2 w_a(t,x_1,x_2) + 2\pi^2 w_a(T,x_1,x_2)\,,
\]
and
\[
      y_0(x_1,x_2) := \frac{-1}{2+a}\pi^2 w_a(0,x_1,x_2)\,.
\]
The admissible set $\Uad$ is defined by the bounds $a_1:=-25$ and $b_1:=-1$. Furthermore $\alpha:=\pi^{-4}$ and $a:=-\sqrt{5}$.

The exact solution of the optimal control problem \eqref{OCP} is given by
\[
      \uopt(t) = P_\Uad\left(
      -\frac{1}{4\alpha}(\exp(a\pi^2 t)-\exp(a\pi^2 T))\right)\,,
\]
\[
      \yopt(t,x_1,x_2) = \frac{-1}{2+a}\pi^2 w_a(t,x_1,x_2)\,,
\]
and
\[
      \popt(t,x_1,x_2) = w_a(t,x_1,x_2)-w_a(T,x_1,x_2)\,.
\]
Note that this example fulfills the Assumption \ref{A:Regularity}.

We solve this problem numerically using a fixpoint iteration on equation \eqref{FONCDiscr}.
We discretize in space with a fixed number of nodes $\text{Nh}=(2^7+1)^2=16\,641$.
We examine the behavior of the temporal convergence by considering a sequence of meshes with $\text{Nk}=(2^\ell+1)^2$ nodes at refinement levels $\ell=1,2,3,4,5,6$. Each fixpoint iteration is initialized by the starting value $u_{kh}:=a_1$.
As a stopping criterion we require
\[
    \| B'\left(p_{kh}^\text{new} - p_{kh}^\text{old}\right) \|_{L^\infty(\Omega\times I)} < t_0\,,
\]
where $t_0:=10^{-5}$ is a prescribed threshold.
\begin{table}[hbt]
\begin{center}
\begin{tabular}{ccccccc}
\hline
$\ell$ & $\|\uopt-u_{kh}\|_{L^1(I,L^1(\Omega))}$ & $\|\uopt-u_{kh}\|_{L^2(L^2)}$ & $\|\uopt-u_{kh}\|_{L^\infty(L^\infty)}$ &
 $\text{EOC}_{L^1}$& $\text{EOC}_{L^2}$ & $\text{EOC}_{L^\infty}$\\
\hline
 1 &  0.07338346 &  0.31701554 &  1.97701729 &    /  &    /   &   / \\
 2 &  0.01653824 &  0.08052755 &  0.66237792 &  2.15 &  1.98  & 1.58\\
 3 &  0.00396507 &  0.01977927 &  0.19440662 &  2.06 &  2.03  & 1.77\\
 4 &  0.00088306 &  0.00448012 &  0.05014900 &  2.17 &  2.14  & 1.95\\
 5 &  0.00017870 &  0.00083749 &  0.00970228 &  2.31 &  2.42  & 2.37\\
 6 &  0.00018581 &  0.00068442 &  0.00462541 &  -0.06&   0.29 &  1.07\\
\hline
\end{tabular}
\end{center}
\caption{First example: Errors and EOC in the control.}
\label{tab:ex1u}
\end{table}
\begin{table}[hbt]
\begin{center}
\begin{tabular}{ccccccc}
\hline
$\ell$ & $\|\yopt-y_{kh}\|_{L^1(I,L^1(\Omega))}$ & $\|\yopt-y_{kh}\|_{L^2(L^2)}$ & $\|\yopt-y_{kh}\|_{L^\infty(L^\infty)}$ &
 $\text{EOC}_{L^1}$& $\text{EOC}_{L^2}$& $\text{EOC}_{L^\infty}$\\
\hline
 1 &  0.19002993 &  0.96055898 &  14.78668742&     /  &    / &     /\\
 2 &  0.09429883 &  0.49287844 &  9.02297459 &  1.01  & 0.96 &  0.71\\
 3 &  0.04706727 &  0.24798983 &  5.06528533 &  1.00  & 0.99 &  0.83\\
 4 &  0.02352485 &  0.12419027 &  2.69722511 &  1.00  & 1.00 &  0.91\\
 5 &  0.01176627 &  0.06216408 &  1.39374255 &  1.00  & 1.00 &  0.95\\
 6 &  0.00588802 &  0.03119134 &  0.70870727 &  1.00  & 0.99 &  0.98\\
\hline
\end{tabular}
\end{center}
\caption{First example: Errors and EOC in the state.}
\label{tab:ex1y}
\end{table}
\begin{table}[hbt]
\begin{center}
\begin{tabular}{ccccccc}
\hline
$\ell$ & $\|\yopt-\pi_{P_k^*} y_{kh}\|_{L^1(I,L^1(\Omega))}$ & $\|\dots\|_{L^2(L^2)}$ & $\|\dots\|_{L^\infty(L^\infty)}$ &
 $\text{EOC}_{L^1}$& $\text{EOC}_{L^2}$& $\text{EOC}_{L^\infty}$\\
\hline
 1 &  0.10937032 &  0.48300664 &  6.29978738 &    /  &    /  &    / \\
 2 &  0.02713496 &  0.13212665 &  1.94510739 &  2.01 &  1.87 &  1.70 \\
 3 &  0.00720221 &  0.03723408 &  0.61346273 &  1.91 &  1.83 &  1.66 \\
 4 &  0.00183081 &  0.00982563 &  0.17399005 &  1.98 &  1.92 &  1.82 \\
 5 &  0.00042588 &  0.00242796 &  0.04646078 &  2.10 &  2.02 &  1.90 \\
 6 &  0.00009796 &  0.00054833 &  0.01201333 &  2.12 &  2.15 &  1.95 \\
\hline
\end{tabular}
\end{center}
\caption{First example: Errors and EOC in the projected state.}
\label{tab:ex1yproj}
\end{table}
\begin{table}[hbt]
\begin{center}
\begin{tabular}{ccccccc}
\hline
$\ell$ & $\|\popt-p_{kh}\|_{L^1(I,L^1(\Omega))}$ & $\|\popt-p_{kh}\|_{L^2(L^2)}$ & $\|\popt-p_{kh}\|_{L^\infty(L^\infty)}$ &
 $\text{EOC}_{L^1}$& $\text{EOC}_{L^2}$& $\text{EOC}_{L^\infty}$\\
\hline
 1 &  0.00125773 &  0.00652756 &  0.08119466 &    /   &   /  &    / \\
 2 &  0.00029007 &  0.00166280 &  0.02721190 &  2.12  & 1.97 &  1.58\\
 3 &  0.00006933 &  0.00040888 &  0.00799559 &  2.06  & 2.02 &  1.77\\
 4 &  0.00001564 &  0.00009340 &  0.00207127 &  2.15  & 2.13 &  1.95\\
 5 &  0.00000310 &  0.00001739 &  0.00041008 &  2.34  & 2.43 &  2.34\\
 6 &  0.00000273 &  0.00001246 &  0.00017768 &  0.18  & 0.48 &  1.21\\
\hline
\end{tabular}
\end{center}
\caption{First example: Errors and EOC in the adjoint.}
\label{tab:ex1p}
\end{table}

Table \ref{tab:ex1u} shows the behavior of several errors in time between the exact control $\uopt$ and its discretized computed counterpart $u_{kh}$, obtained by the fixpoint iteration. Furthermore, the \emph{experimental order of convergence} (EOC) is given. The table indicates an error behavior of $\mathcal{O}(k^2)$ for the $L^2$ error in the control, which is in accordance with Theorem \ref{T:main}. Furthermore, the error of the adjoint, see table \ref{tab:ex1p}, shows the same behavior, as expected. Here we note that the EOC in our numerical example deteriorates if the temporal error reaches the size of the spatial error (which in the numerical investigations is fixed through the choice of $\text{Nh}$), see e.g. the last lines in Tables \ref{tab:ex1u}, \ref{tab:ex1p}, \ref{tab:ex2u}, \ref{tab:ex2p}.\\
Since the state is discretized piecewise constant in time, the order of convergence is only one. This is depicted in table \ref{tab:ex1y}. However, without further numerical effort we obtain a second order convergent approximation of the state with the projection $\pi_{P_k^*}y_k$ of the discrete state $y_k$, see Theorem \ref{T:2ndorderstate} and see table \ref{tab:ex1yproj} for the corresponding numerical results. In practise this means that we can gain a better approximation of the state without further effort; we only have to interpret the discrete state vector $\vec y_k$, i.e. the vector containing the value of $y_k$ on each interval $I_\j$, in the right way, namely as a vector of values on the gridpoints of the dual grid $t_1^* < \dots < t_\K^*$.

\begin{figure}
    \centering  
   \subfloat[$\ell=1$]{ \includegraphics[trim=25mm 75mm 15mm 91mm,clip,width=0.3\textwidth]{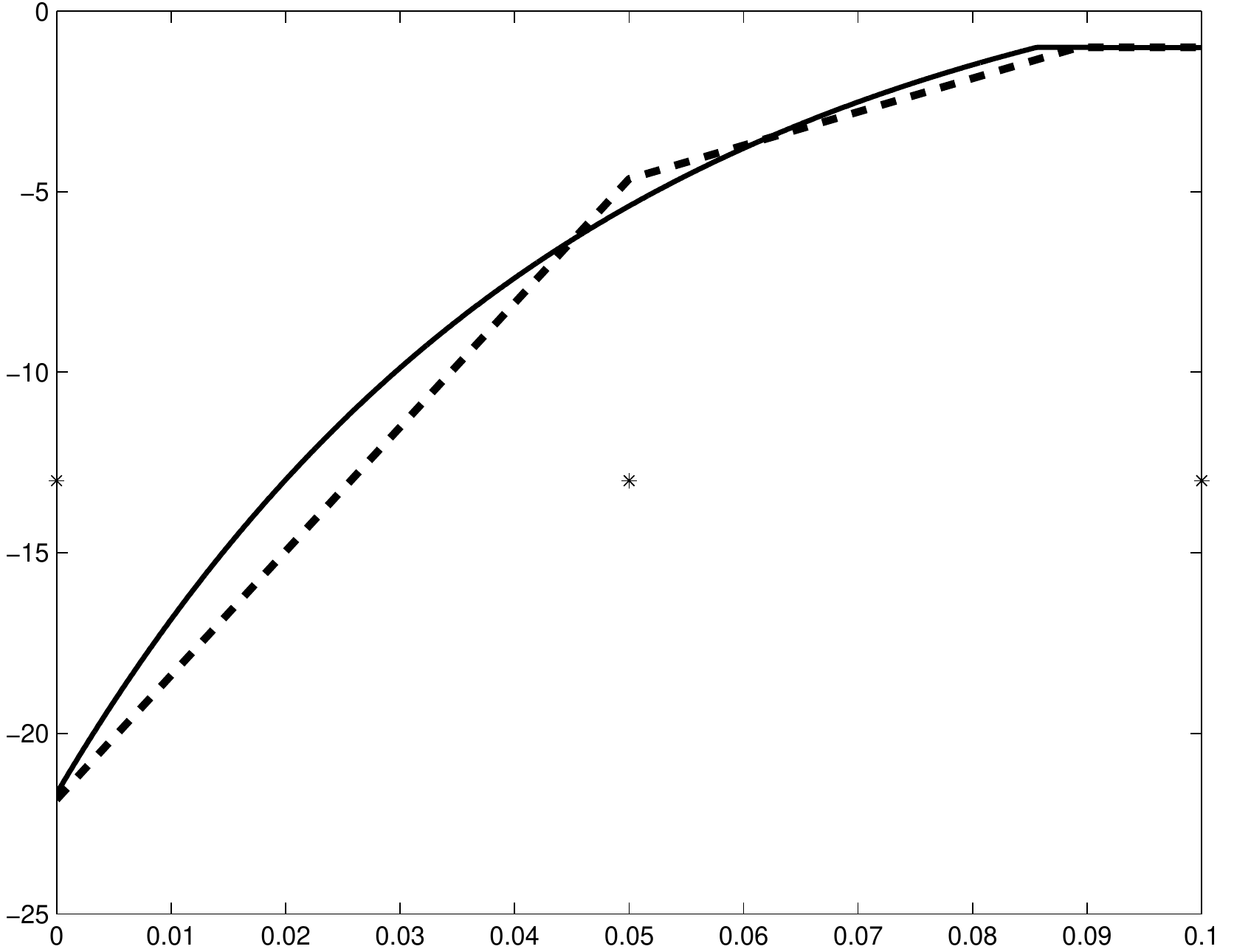}  }
   \subfloat[$\ell=2$]{ \includegraphics[trim=25mm 75mm 15mm 91mm,clip,width=0.3\textwidth]{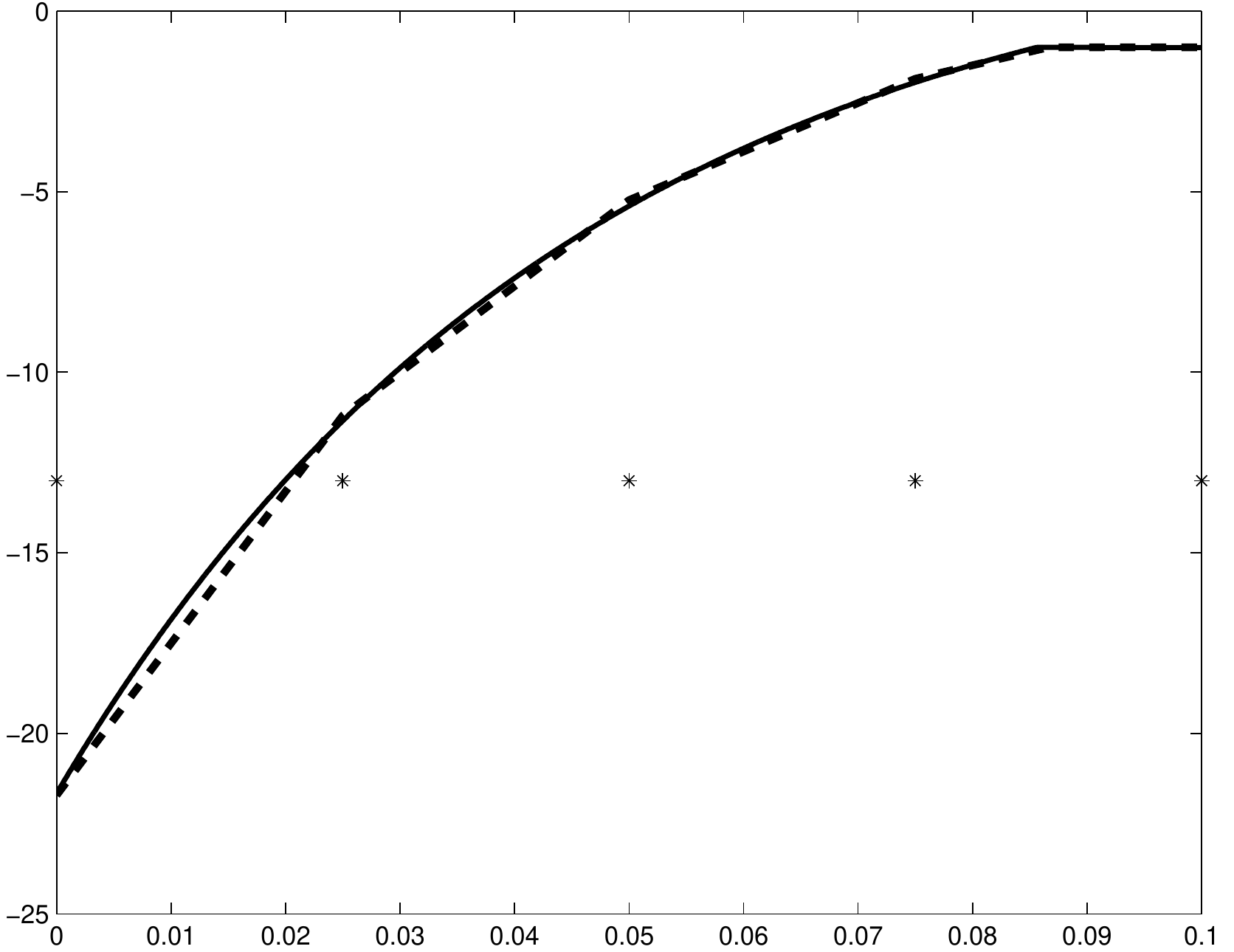}  }
   \subfloat[$\ell=3$]{ \includegraphics[trim=25mm 75mm 15mm 91mm,clip,width=0.3\textwidth]{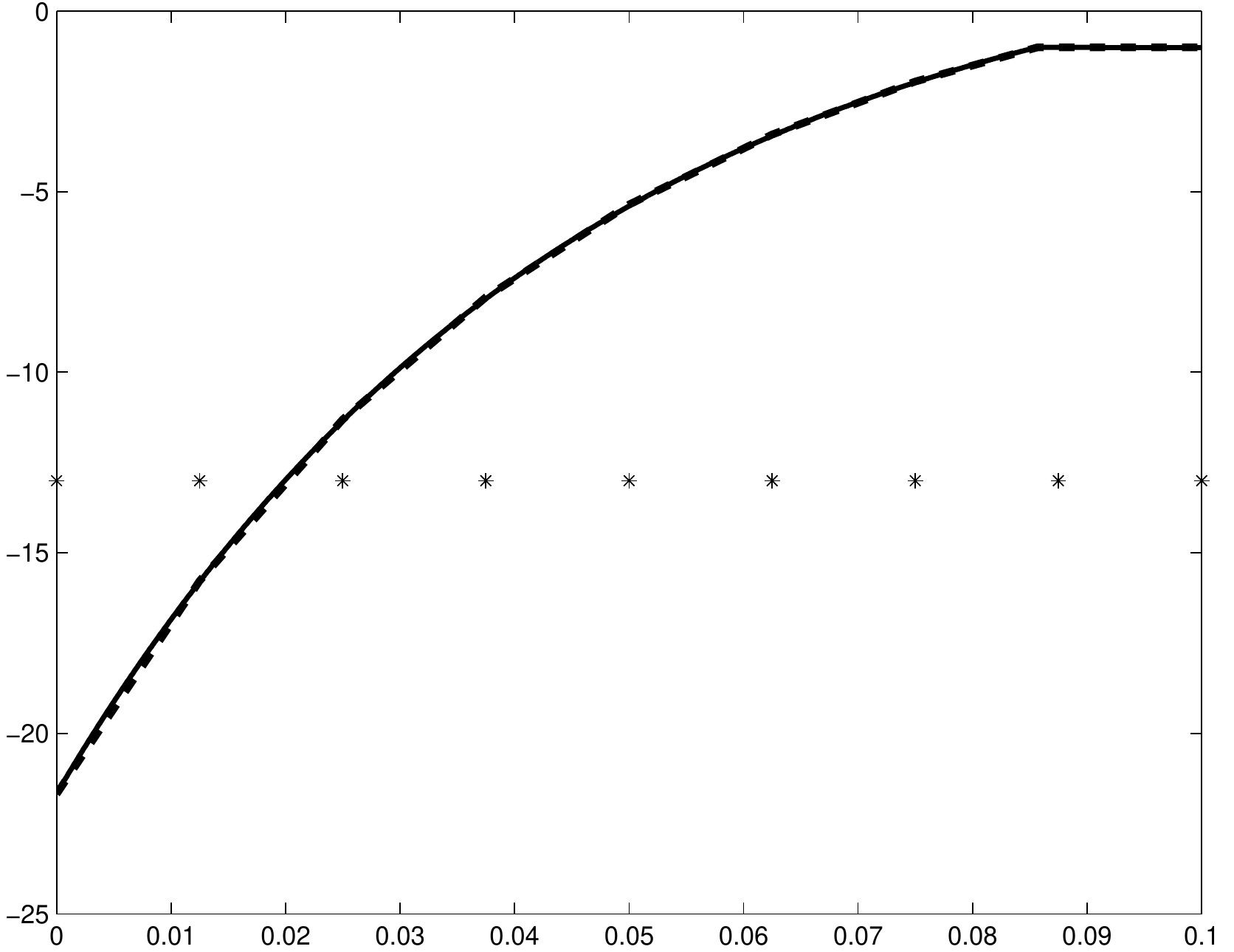}  }
    \caption{First example: Optimal control $\uopt$ (solid) and $u_{kh}$ (dashed) over time after refinement level $\ell$.}
    \label{fig:ex1}
\end{figure}

Figure \ref{fig:ex1} illustrates the convergence of $u_{kh}$ to $\uopt$. Note that the intersection points between inactive set $\mathcal I_{kh}=\set{t\in I}{a < u_{kh}(t) < b}$ and active set $\mathcal A_{kh}:=I\backslash \mathcal I_{kh}$ need not coincide with time grid points since we use variational discretization.
Let us further note that the number of fixpoint iterations does not depend on the fineness of the time grid size. In our example four iterations are needed to reach the above mentioned threshold of $t_0:=10^{-5}$. Let us mention that the fixpoint iteration only converges for large enough values of $\alpha$, see e.g. \cite{HinzeVierling2012} which seems to be the case in our numerical examples. For smaller values of $\alpha$ the semi-smooth Newton method can be applied, see also \cite{HinzeVierling2012} for its numerical analysis in the case of variational discretization of elliptic optimal control problems. 

\subsection{Second example}

This example is a slight variant of the first one yielding more intersection points between the active and inactive set.
With the space-time domain $\Omega\times I = (0,1)^2 \times (0,0.5)$, we set
\[
      \yopt(t,x_1,x_2) := w_a(t,x_1,x_2) := \cos\left(\frac tT\,2\pi a\right)\cdot g_1(x_1,x_2)\,,
\]
\[
      \popt(t,x_1,x_2) := w_a(t,x_1,x_2)-w_a(T,x_1,x_2)\,,
\]
\[
      y_0(x_1,x_2) := g_1(x_1,x_2)\,,
\]
where $g_1$ is defined in the first example.
Consequently,
\[
      g_0= g_1 2 \pi \left( -\frac aT \sin\left(\frac tT\,2\pi a\right) + \pi\cos\left(\frac tT\,2\pi a\right) \right)-B\uopt\,,
\]
\[
      y_d= g_1 \left( \cos\left(\frac tT\,2\pi a\right) \left(1-2\pi^2\right) - \frac{2\pi a}{T} \sin\left(\frac tT\,2\pi a\right) +2\pi^2 \cos\left(2\pi a\right) \right)\,,
\]
and
\[
      \uopt= P_\Uad\left( -\frac{1}{4\alpha} \cos\left(\frac tT\,2\pi a\right) + \frac{1}{4\alpha} \right)\,.
\]
Futhermore, we set $\alpha=1$, $a_1:=0.2$, $b_1:=0.4$ and $a:=2$. Note that this example also fulfills the Assumption \ref{A:Regularity}.

We now consider refinement levels $\ell=1,2,3,4,5,6,7,8$ and proceed as described in the first example. We obtain the same EOCs for control, state, and adjoint, see the tables \ref{tab:ex2u}, \ref{tab:ex2y}, \ref{tab:ex2yproj}, \ref{tab:ex2p}, and figure \ref{fig:ex2}.
\begin{table}[hbt]
\begin{center}
\begin{tabular}{ccccccc}
\hline
$\ell$ & $\|\uopt-u_{kh}\|_{L^1(I,L^1(\Omega))}$ & $\|\uopt-u_{kh}\|_{L^2(L^2)}$ & $\|\uopt-u_{kh}\|_{L^\infty(L^\infty)}$ &
 $\text{EOC}_{L^1}$& $\text{EOC}_{L^2}$ & $\text{EOC}_{L^\infty}$\\
\hline
 1 &  0.04925427 &  0.09237138 &  0.20000000 &    /   &   /    &  / \\
 2 &  0.00256632 &  0.01106114 &  0.07336869 &  4.26  & 3.06   &1.45\\
 3 &  0.00403215 &  0.01144324 &  0.04704583 &  -0.65 &  -0.05 &  0.64\\
 4 &  0.00069342 &  0.00204495 &  0.00893696 &  2.54  & 2.48   &2.40\\
 5 &  0.00016762 &  0.00050729 &  0.00249463 &  2.05  & 2.01   &1.84\\
 6 &  0.00003989 &  0.00011939 &  0.00064497 &  2.07  & 2.09   &1.95\\
 7 &  0.00000948 &  0.00003227 &  0.00020672 &  2.07  & 1.89   &1.64\\
 8 &  0.00000764 &  0.00002142 &  0.00009457 &  0.31  & 0.59   &1.13\\
\hline
\end{tabular}
\end{center}
\caption{Second example: Errors and EOC in the control.}
\label{tab:ex2u}
\end{table}
\begin{table}[hbt]
\begin{center}
\begin{tabular}{ccccccc}
\hline
$\ell$ & $\|\yopt-y_{kh}\|_{L^1(I,L^1(\Omega))}$ & $\|\yopt-y_{kh}\|_{L^2(L^2)}$ & $\|\yopt-y_{kh}\|_{L^\infty(L^\infty)}$ &
 $\text{EOC}_{L^1}$& $\text{EOC}_{L^2}$& $\text{EOC}_{L^\infty}$\\
\hline
 1 &  0.19657193 &  0.41315218 &  2.24553307 &    /  &    /  &    / \\
 2 &  0.13005269 &  0.25408123 &  1.25552256 &  0.60 &  0.70 &  0.84\\
 3 &  0.05650537 &  0.11224959 &  0.65977254 &  1.20 &  1.18 &  0.93 \\
 4 &  0.02611675 &  0.05637041 &  0.38210207 &  1.11 &  0.99 &  0.79 \\
 5 &  0.01277289 &  0.02827337 &  0.19029296 &  1.03 &  1.00 &  1.01 \\
 6 &  0.00635223 &  0.01418903 &  0.09710641 &  1.01 &  0.99 &  0.97 \\
 7 &  0.00317298 &  0.00718111 &  0.04892792 &  1.00 &  0.98 &  0.99 \\
 8 &  0.00158730 &  0.00375667 &  0.02456764 &  1.00 &  0.93 &  0.99 \\
\hline
\end{tabular}
\end{center}
\caption{Second example: Errors and EOC in the state.}
\label{tab:ex2y}
\end{table}

\begin{table}[hbt]
\begin{center}
\begin{tabular}{ccccccc}
\hline
$\ell$ & $\|\yopt-\pi_{P_k^*} y_{kh}\|_{L^1(I,L^1(\Omega))}$ & $\|\dots\|_{L^2(L^2)}$ & $\|\dots\|_{L^\infty(L^\infty)}$ &
 $\text{EOC}_{L^1}$& $\text{EOC}_{L^2}$& $\text{EOC}_{L^\infty}$\\
\hline
 1 &  0.19734452 &  0.42154165 &  2.65669891 &    /  &    /  &    / \\
 2 &  0.13173168 &  0.25800727 &  1.39668789 &  0.58 &  0.71 &  0.93\\
 3 &  0.03422500 &  0.07418402 &  0.40783930 &  1.94 &  1.80 &  1.78 \\
 4 &  0.01080693 &  0.02168391 &  0.15176831 &  1.66 &  1.77 &  1.43 \\
 5 &  0.00282859 &  0.00567595 &  0.04685968 &  1.93 &  1.93 &  1.70 \\
 6 &  0.00071212 &  0.00143268 &  0.01229008 &  1.99 &  1.99 &  1.93 \\
 7 &  0.00017551 &  0.00035509 &  0.00311453 &  2.02 &  2.01 &  1.98 \\
 8 &  0.00004104 &  0.00008530 &  0.00078765 &  2.10 &  2.06 &  1.98 \\
\hline
\end{tabular}
\end{center}
\caption{Second example: Errors and EOC in the projected state.}
\label{tab:ex2yproj}
\end{table}
\begin{table}[hbt]
\begin{center}
\begin{tabular}{ccccccc}
\hline
$\ell$ & $\|\popt-p_{kh}\|_{L^1(I,L^1(\Omega))}$ & $\|\popt-p_{kh}\|_{L^2(L^2)}$ & $\|\popt-p_{kh}\|_{L^\infty(L^\infty)}$ &
 $\text{EOC}_{L^1}$& $\text{EOC}_{L^2}$& $\text{EOC}_{L^\infty}$\\
\hline
 1 &  0.20659855 &  0.46853028 &  2.86360259 &    /  &    /  &    / \\
 2 &  0.03491931 &  0.08118048 &  0.56829981 &  2.56 &  2.53 &  2.33\\
 3 &  0.01994220 &  0.04100552 &  0.20495644 &  0.81 &  0.99 &  1.47\\
 4 &  0.00440890 &  0.00895349 &  0.05815307 &  2.18 &  2.20 &  1.82\\
 5 &  0.00105993 &  0.00215639 &  0.01668075 &  2.06 &  2.05 &  1.80\\
 6 &  0.00026116 &  0.00053258 &  0.00447036 &  2.02 &  2.02 &  1.90\\
 7 &  0.00006984 &  0.00014824 &  0.00116014 &  1.90 &  1.85 &  1.95\\
 8 &  0.00004199 &  0.00008530 &  0.00046798 &  0.73 &  0.80 &  1.31\\
\hline
\end{tabular}
\end{center}
\caption{Second example: Errors and EOC in the adjoint.}
\label{tab:ex2p}
\end{table}
\begin{figure}
    \centering  
   \subfloat[$\ell=1$]{ \includegraphics[trim=25mm 75mm 15mm 91mm,clip,width=0.3\textwidth]{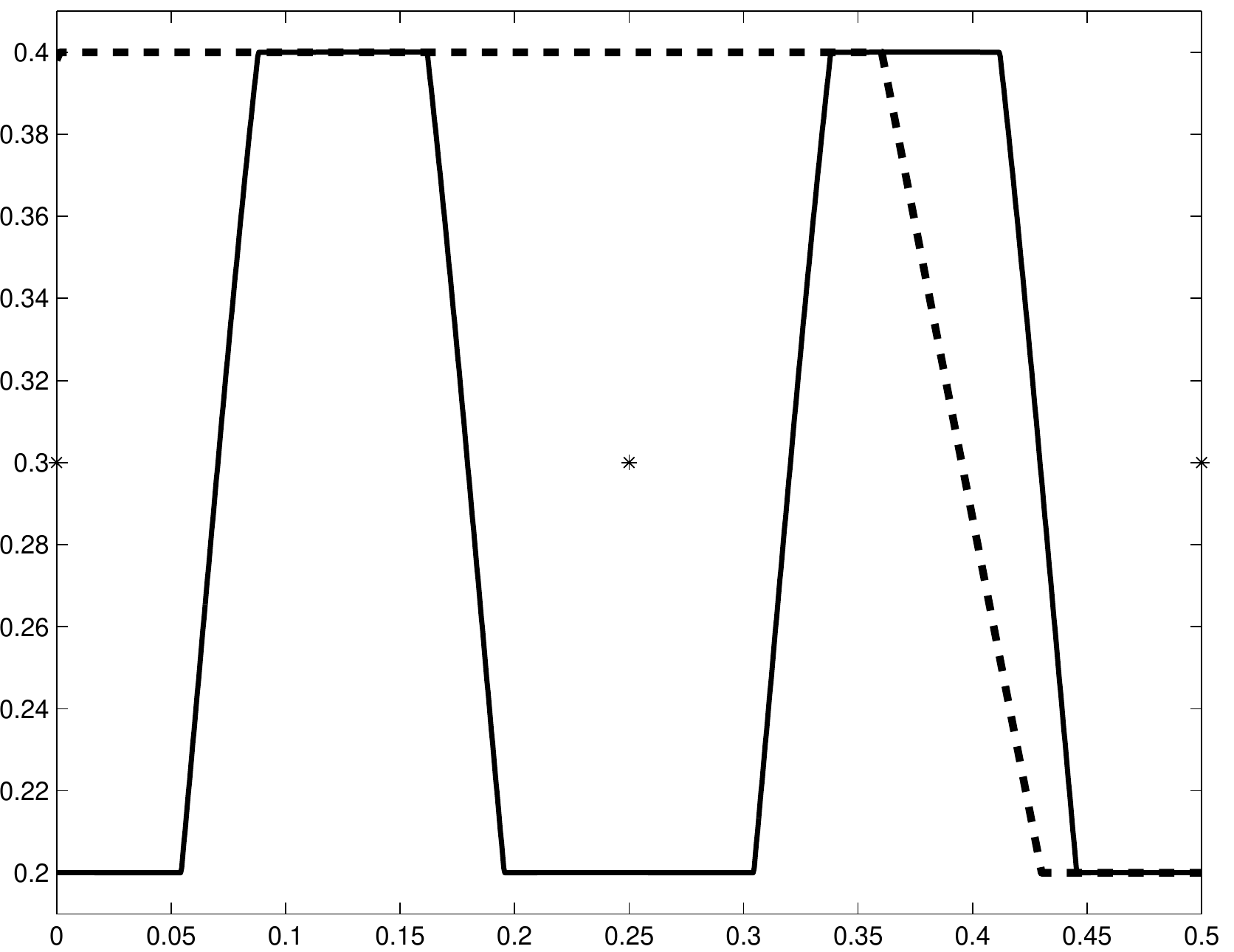}  }
   \subfloat[$\ell=2$]{ \includegraphics[trim=25mm 75mm 15mm 91mm,clip,width=0.3\textwidth]{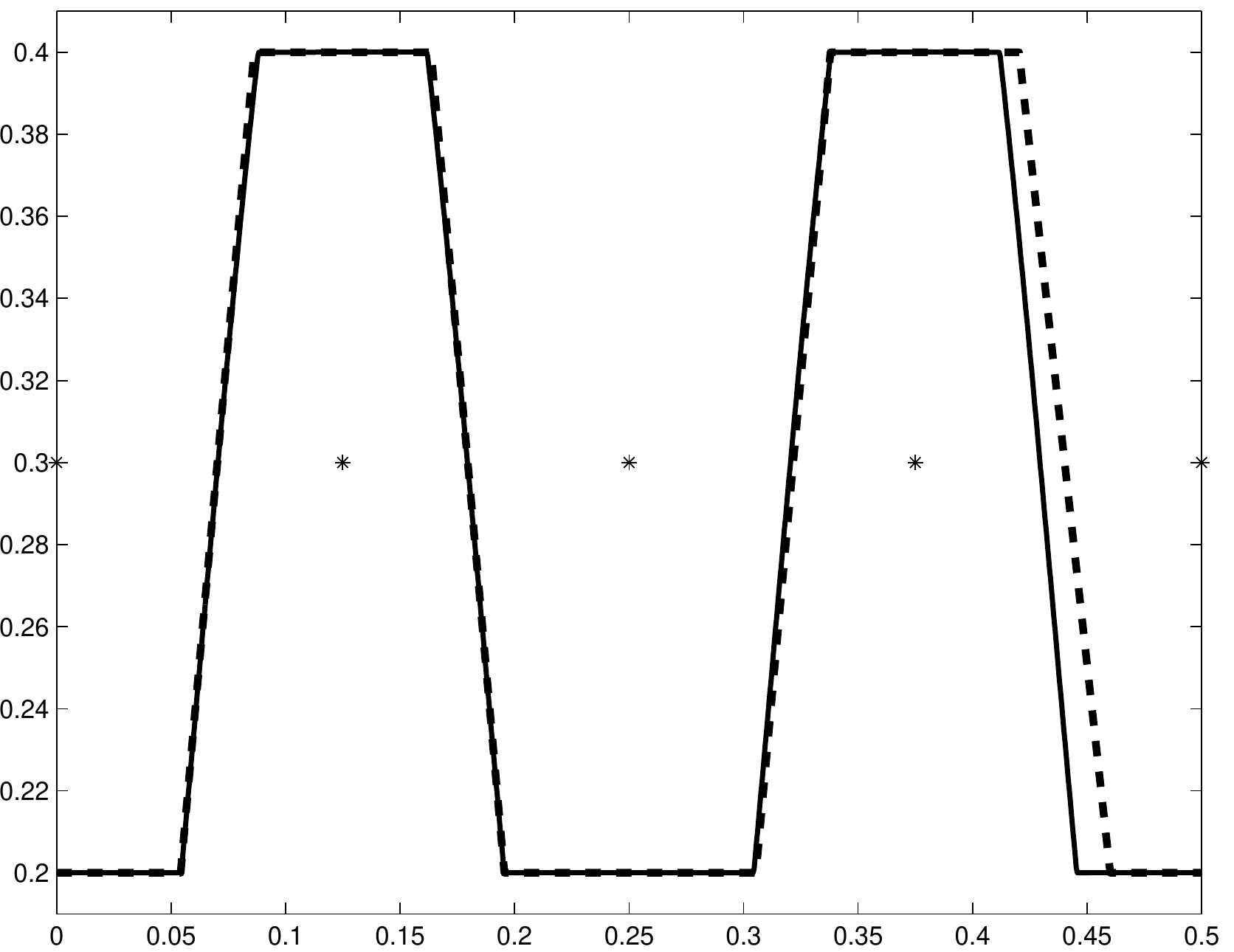}  }
   \subfloat[$\ell=3$]{ \includegraphics[trim=25mm 75mm 15mm 91mm,clip,width=0.3\textwidth]{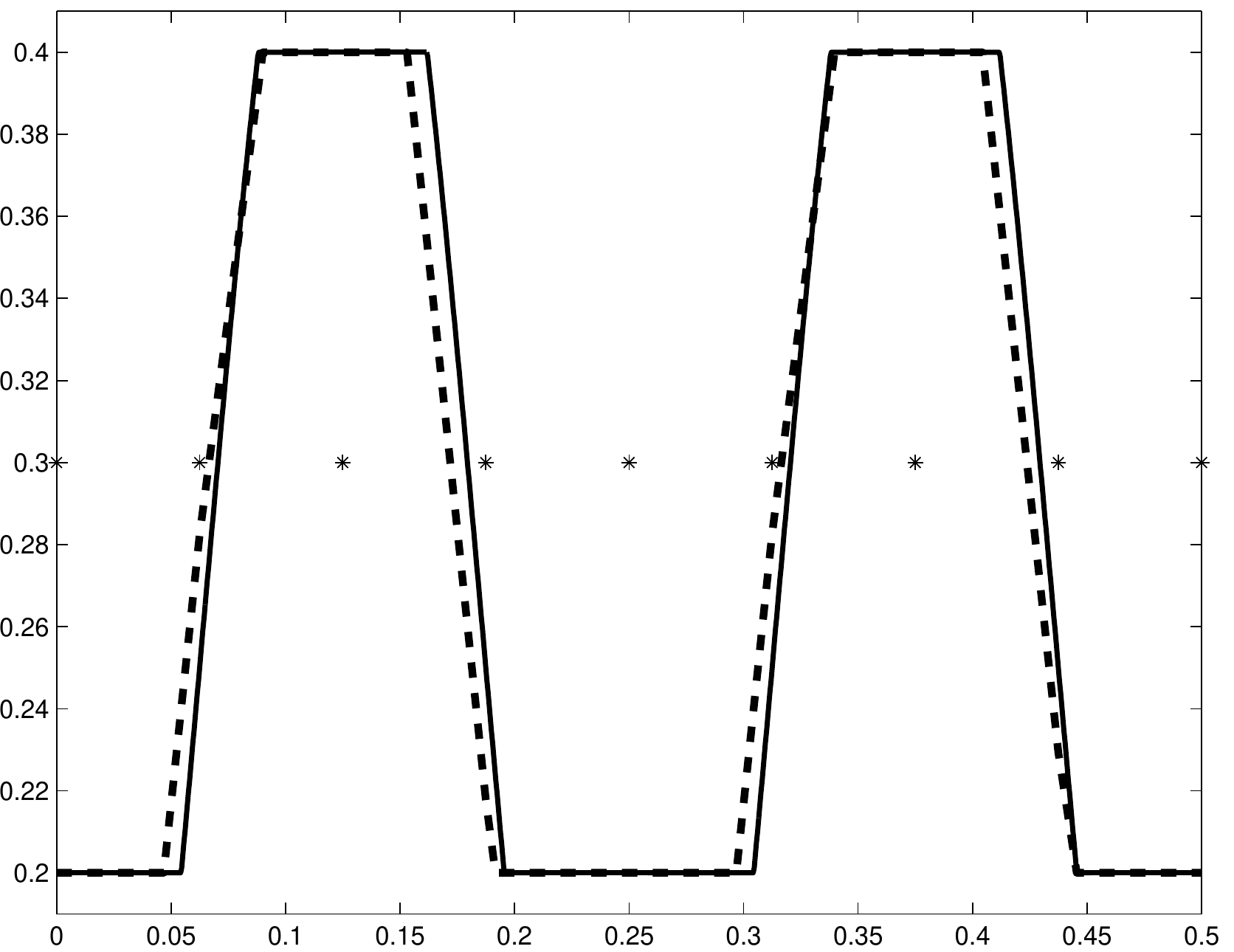}  }
    \caption{Second example: Optimal control $\uopt$ (solid) and $u_{kh}$ (dashed) over time after refinement level $\ell$.}
    \label{fig:ex2}
\end{figure}

\section*{Acknowledgements}

Christian Kahle provided some hints on the numerical implementation which are gratefully acknowledged.

\bibliographystyle{annotate}

\end{document}